\documentclass{amsart}
\usepackage{amsmath, amssymb, amsthm, mathrsfs, setspace}
\usepackage[dvips]{color} 
\usepackage[dvips]{graphicx}

\newtheorem{theorem}{Theorem}[section]
\newtheorem{lemma}[theorem]{Lemma}
\newtheorem{corollary}[theorem]{Corollary}
\newtheorem{proposition}[theorem]{Proposition}
\newtheorem{definition}[theorem]{Definition}

\newtheorem*{acknowledge}{Aknowledgments}
\theoremstyle{remark}
\newtheorem{remark}[theorem]{Remark}

\newcommand{\del}{\partial}

\newcommand{\Z}{\mathbb{Z}}

\newcommand{\Mod}{{\rm{Mod}}}

\begin{document}

\title[Stein fillings of homology $3$-spheres and Mapping class groups]{Stein fillings of homology $3$-spheres and \\ Mapping class groups }

\author[Takahiro Oba]{Takahiro Oba}
\address{Department of Mathematics, Tokyo Institute of Technology, 2-12-1 Ookayama, Meguroku, Tokyo 152-8551, Japan}
\email{oba.t.ac@m.titech.ac.jp}

\begin{abstract}
In this article, using combinatorial techniques of mapping class groups, we show that 
a Stein fillable integral homology $3$-sphere supported by an open book decomposition with page a $4$-holed sphere 
admits a unique Stein filling up to diffeomorphism.
Furthermore, according to a property of deforming symplectic fillings of a rational homology $3$-spheres into strongly symplectic fillings, 
we also show that a symplectically fillable integral homology $3$-sphere 
supported by an open book decomposition with page a $4$-holed sphere 
admits a unique symplectic filling up to diffeomorphism and blow-up.
 
\end{abstract}


\date{\today}


\maketitle
  
   \section{Introduction.}
   A $2$-plane field $\xi$ on an oriented $3$-manifold $M$ is called an (\textit{oriented}) \textit{contact structure} on $M$ if 
   there exists a $1$-form $\alpha$ such that $\xi = {\rm{ker} \alpha}$ and $\alpha \wedge d \alpha>0$ with respect to 
   the orientation of $M$.
	The pair $(M, \xi)$ is called a \textit{contact manifold}.
   Recently, Giroux \cite{Gi} gave a one-to-one correspondence 
	between the set of contact structures on $M$ and 
	the set of open book decompositions of $M$.
   This correspondence enables us to handle contact $3$-manifolds combinatorially by surface topology and mapping class groups.
    
	A $2$-form $\omega$ on a $4$-manifold $X$ is a \textit{symplectic structure} on $X$ if 
	$\omega$ is closed and non-degenerate.
	The pair $(X, \omega)$ is called a \textit{symplectic manifold}.
	A (\textit{weak}) \textit{symplectic filling} of $(M, \xi)$ is 
	 a compact symplectic $4$-manifold $(X, \omega)$ such that $\del W = M $ as oriented manifolds and  $\omega | \xi > 0$.
	 A \textit{strongly symplectic filling} of $(M, \xi)$ is a compact symplectic $4$-manifold $(X, \omega)$ 
	 such that $\omega$ is exact near $\del X = M$ and a $1$-form $\alpha$ with $d \alpha = \omega$ 
	can be chosen so that ${\rm{ker}} (\alpha| \del X) = \xi$.
	A \textit{Stein filling} of $(M, \xi)$ is a complex surface $(X, J)$ 
	with a proper plurisubharmonic function $\Phi : X \rightarrow [0,\infty)$ such that 
	for a regular value $t$ the 3-manifold $M_{t} = \Phi ^{-1}(t)$ with the 2-plane field of complex tangencies $TM_{t}\cap JTM_{t}$ is 
	contactmorphic to $(M, \xi)$.
	A contact structure $\xi$ on $M$ is said to be (\textit{weakly}) \textit{symplectically fillable}, 
	\textit{strongly symplectically fillable}, and \textit{Stein fillable}  
	if $(M, \xi)$ admits a symplectic filling, strongly symplectic filling, and Stein filling respectively.
	It is a well-known fact after Eliashberg and Gromov \cite{EG} that a fillable contact structure on a $3$-manifold is tight.
	Thus various fillings play important roles in examining topology of $(M, \xi)$.
	
	Classification of these fillings of a given contact $3$-manifold has been 
	discussed as a crucial problem in contact and symplectic geometry. 
	In particular, it is important to determine which contact $3$-manifold admits a unique filling.
	Eliashberg \cite{El} showed that the $3$-sphere $S^3$ with the standard contact structure $\xi_{std}$
	admits a unique symplectic filling up to symplectic deformation and blow-up. 
	McDuff \cite{Mc}  showed that 
	the lens space $L(p,1)$ for $p\neq 4$ with the standard contact structure admits
	a unique symplectic filling up to diffeomorphism and blow-up, 
	and Hind \cite{Hi} showed that it admits a unique Stein filling up to Stein deformation.
	Furthermore, 
	Plamenevskaya and Van Horn-Morris \cite{PV} showed that 
	every tight contact structure on $L(p,1)$ for $p\neq 4$ admits a unique symplectic filling up to symplectic deformation and blow-up.
	Lisca \cite{L} and Kaloti \cite{K}  showed that there are infinitely many tight lens spaces different from $L(p,1)$ 
	each of which admits a unique symplectic filling up to diffeomorphism and blow-up. 
	On other $3$-manifolds, 
	Ohta and Ono \cite{OO2} showed that the links of simple singularities with the standard contact structures 
	admit unique symplectic fillings up to symplectic deformation and blow-up.
	Stipsicz \cite{S} showed that the Poincar\'{e} homology $3$-sphere with a Stein fillable contact structure 
	and the $3$-torus $T^3$ with the unique Stein fillable contact structure 
	admit unique Stein fillings up to homeomorphism.
	Wendl \cite{W} showed that the same contact $T^3$ admits a unique strongly filling up to symplectic deformation and blow-up.
	Sch\"{o}nenberger \cite{Sch}  constructed an infinite family of contact small Seifert $3$-manifolds 
	each of which admits a unique symplectic filling up to diffeomorphism and blow-up.
	Kaloti and Li \cite{KL} constructed one of contact $3$-manifolds including infinitely many hyperbolic $3$-manifolds 
	such that each member admits a unique Stein filling up to symplectic deformation.
	
	In this article, we will give a new family of contact $3$-manifolds each of which admits 
	a unique Stein filling up to diffeomorphism.
	This family has a different flavor from the above earlier results.
	We emphasize that our uniqueness result hold for 
	all Stein fillable integral homology $3$-spheres 
	supported by open book decompositions with page a $4$-holed sphere $\Sigma_{0,4}$.
	
	 \begin{theorem}\label{thm: uniqueness}
	 Let $M$ be an integral homology $3$-sphere.
	Suppose that a contact structure $\xi$ on $M$ is Stein fillable and 
	supported by an open book decomposition with page $\Sigma_{0,4}$.
	Then the contact $3$-manifold $(M,\xi)$ admits a unique Stein filling up to diffeomorphism.
	Furthermore this Stein filling is diffeomorphic to either the $4$-disk $D^4$ or a Mazur type manifold.
   \end{theorem}
   
   Here a $4$-manifold $X$ is called of \textit{Mazur type} if 
	it is contractible, the boundary $\del X$ is not diffeomorphic to $S^3$, and 
	it admits a handle decomposition consisting of one $0$-handle, one $1$-handle, and one $2$-handle.
	
	\begin{remark}
	There exist infinitely many integral homology $3$-spheres satisfying the assumption of Theorem \ref{thm: uniqueness}.
	In fact, the author constructed such a family in \cite{O}.
   \end{remark}
   
	The earlier uniqueness results except the ones of Plamenevskaya and Van Horn-Morris, Kaloti and Li, and Kaloti  
	were shown by techniques of symplectic geometry and gauge theory.
	The three results and ours, however, are shown by more combinatorial techniques of mapping class groups.
	These are based on the following work of Wendl \cite{W}.
   
   \begin{theorem}[Wendl {\cite[Theorem $1$]{W}}]\label{thm: Wendl}
   Suppose that $(X, \omega)$ is a strongly symplectic filling of a contact $3$-manifold $(M,\xi)$ 
	supported by a planar open book decomposition $(\Sigma, \varphi)$.
	Then $(W, \omega)$ is symplectic deformation equivalent to 
	a blow-up of the total space of a positive allowable Lefschetz fibration 
	obtained from $(\Sigma, \varphi)$.
	\end{theorem}
	
	Note that a Stein filling $(X, J)$ of a contact $3$-manifold $(M, \xi)$ is minimal, and thus by the theorem, 
	$X$ is diffeomorphic to the total space of a positive allowable Lefschetz fibration obtained from 
	a given open book decomposition.
	Hence to classify Stein fillings of $(M, \xi)$ we examine the \textit{positive factorizations} of 
	the monodromy of the open book decomposition.
	
	Furthermore, Ohta and Ono \cite{OO} showed that if $M$ is a rational homology $3$-sphere, 
	then a symplectic filling of a contact $3$-manifold $(M, \xi)$ can be deformed its strongly symplectic filling.
	The next corollary follows from this fact immediately.

      \begin{corollary}\label{cor}
	 Let $M$ be an integral homology $3$-sphere.
	Suppose that a contact structure $\xi$ on $M$ is symplectically fillable and 
	supported by an open book decomposition with page $\Sigma_{0,4}$.
	Then the contact $3$-manifold $(M,\xi)$ admits a unique symplectic filling up to diffeomorphism and blow-up.
   \end{corollary}
   
   This article is constructed as follows: 
   In Section $2$, we review some definitions and properties of mapping class groups, positive Lefschetz fibrations and 
	open book decompositions.
	In Section $3$, 
	first, we examine a handle decomposition of a Stein filling of an integral homology $3$-sphere 
	using some properties of planar open book decompositions.
	Next, according to this observation, we prove several propositions about the number of handles of the given handle decomposition.
	Finally, we prove the main results using these propositions, Kirby calculus and mapping class groups.
	
    Throughout this article we will work in the smooth category.
    We assume that the reader is familiar with a few basics of contact geometry, Kirby diagrams and Kirby calculus.
    If necessary, we refer the reader to \cite{E}, \cite{Ge}, and \cite{OS} about the former 
	and to \cite[Section $4$ and $5$]{GS} about the latter.

    \begin{acknowledge}
    \rm{The author would like to express his gratitude to Professor Hisaaki Endo for his encouragement and many helpful comments for the draft of this article. 
    He would also like to thank Kouichi Yasui for his kindly comments on results in this article.}
    \end{acknowledge}

   \section{Preliminaries.}
   
   \subsection{Mapping class groups.}\label{MCG}
   
	Let $\Sigma$ be a compact, oriented, connected surface with boundary.
	
	\begin{definition}
	The \textit{mapping class group} ${\rm{Mod}}(\Sigma)$ of $\Sigma$ is the group of isotopy classes of 
	orientation-preserving diffeomorphisms of $\Sigma$ relative to the boundary $\del \Sigma$.
	\end{definition}
	
	Let $\alpha$ be a simple closed curve in $\Sigma$. 
	A \textit{right-handed Dehn twist} $t_{\alpha}: \Sigma \rightarrow \Sigma$ along 
	$\alpha$ in $\Sigma$ is a diffeomorphism obtained by cutting $\Sigma$ 
	along $\alpha$, twisting $360^{\circ}$ to the right and regluing.
	The mapping class of $t_{\alpha}$ is also denoted by $t_{\alpha}$ and also called the right-handed Dehn twist along $\alpha$.

	Let $\alpha$ and $\beta$ be simple closed curves in $\Sigma$.
	The \textit{geometric intersection number} between $\alpha$ and $\beta$ is denoted by $i(\alpha, \beta)$.
	We say that the collection $\{\alpha, \beta \}$ \textit{fills} $\Sigma$
	if, for any simple closed curve $\gamma$ in $\Sigma$, it holds that $i(\gamma, \alpha)>0$ or $i(\gamma, \beta)>0$.
	The subgroups of $\Mod (\Sigma)$ generated by two right-handed Dehn twists are 
	classified by the geometric intersection number between the simple closed curves generating these Dehn twists 
	and isotopy classes of these curves. 
	In particular, Ishida \cite[Theorem $1.2$]{I} showed that if $i(\alpha, \beta) \geq 2$, 
	the subgroup $\langle t_{\alpha}, t_{\beta} \rangle$  of $\Mod (\Sigma)$ generated by $t_{\alpha}$ and $t_{\beta}$ 
	is isomorphic to the free group $F_{2}$ of rank $2$.
	If $\{\alpha, \beta \}$ fills $\Sigma$, 
	we define the representation $\varrho: \langle t_{\alpha}, t_{\beta} \rangle \rightarrow PSL(2; \Z)$ as 
	\[
		t_{\alpha} \mapsto 
		\begin{pmatrix}
			1 & -i(\alpha, \beta) \\
			0 & 1 
		\end{pmatrix},\ 
		t_{\beta} \mapsto 
		\begin{pmatrix}
			1 & 0 \\
			i(\alpha, \beta) & 1 
		\end{pmatrix}.
	\]
	Note that if $i(\alpha, \beta ) \geq 2$, the representation $\varrho$ is injective.
	More precisely about the representation, we refer to the reader to \cite[Section $14.1.2$]{FM}.
   
   \subsection{Positive Lefschetz fibrations and open book decompositions.}\label{section: LF and OB}
   Let $X$ and $B$ be compact, oriented, connected smooth $4$- and $2$-manifolds.
   
   \begin{definition}\label{def: LF}
	A smooth map $f:X\rightarrow B$ is called a \textit{positive Lefschetz fibration} if there exist points
	$b_{1}, b_{2}, \dots , b_{m}$ in the interior of $B$ such that
	\begin{enumerate}
		\item $f|f^{-1}(B-\{b_{1}, b_{2}, \dots , b_{m}\}):f^{-1}(B-\{b_{1}, b_{2}, \dots , b_{m}\}) \rightarrow B-\{b_{1}, b_{2}, \dots , b_{m}\}$
				is a smooth fiber bundle over $B-\{b_{1}, b_{2}, \dots , b_{m}\}$ with fiber an oriented surface $\Sigma$,
		\item $\{ b_{1}, b_{2}, \dots , b_{m} \}$ are the critical values of $f$, with $p_{i} \in f^{-1}( b_{i})$ a unique critical point of $f$ for each $i$, 
		\item for each $b_{i}$ and $p_{i}$, there are local complex coordinate charts with respect to the orientations of $X$ 
		and $B$ such that locally f can be written as $f(z_{1},z_{2}) = z_{1}^2 + z_{2}^2$, and
		\item no fiber contains a $(-1)$-sphere, that is, an embedded sphere with self-intersection number $-1$.
	\end{enumerate}
   \end{definition}
   
	We call a fiber $f^{-1}(b)$ a \textit{singular fiber} if $b\in \{b_{1}, b_{2}, \dots , b_{m}\}$ and a \textit{regular fiber} otherwise. 
	Also we call $X$ the \textit{total space} and $B$ the \textit{base space}.
	
	We will review roughly a handle decomposition of the total space and 
	a monodromy of a given positive Lefschetz fibration over the disk $D^2$. 
	For more details we refer the reader to \cite[Section $8. 2$]{GS} and \cite[Section $10. 1$]{OS}.
	Suppose $f:X\rightarrow D^2$ is a positive Lefschetz fibration with fiber 
	a compact, oriented, connected genus $g$ surface $\Sigma$ with $r$ boundary components. 
	Then $X$ admits a handle decomposition 
	\[
		X  = (D^2 \times \Sigma) \cup (\cup_{i=1}^{m}h^{(2)}_{i})
		    =   (h^{(0)} \cup (\cup_{j=1}^{2g+r-1} h^{(1)}_{j})) \cup (\cup_{i=1}^{m}h^{(2)}_{i})  
	\]
	where each $h^{(k)}_{l}$ is a $k$-handle, and each $2$-handle $h_{i}^{(2)}$ corresponding to the critical point $p_{i}$ is 
	attached along a simple closed curve $\alpha_{i}$ in $\{ pt. \} \times \Sigma \subset D^2 \times \Sigma$ with 
	framing $-1$ relative to the product framing on $\alpha_{i}$. 
	The attaching circle $\alpha_{i}$ of $h^{(2)}_{i}$ is called a \textit{vanishing cycle} for the singular fiber $f^{-1}(b_{i})$. 
	Then the $m$-tuple $( t_{\alpha_{1}}, \,  t_{\alpha_{2}} ,\cdots, t_{\alpha_{m}})$ of $m$ right-handed Dehn twists 
	is called a \textit{monodromy } of $f$. 
	The product $t_{\alpha_{m}}  t_{\alpha_{m-1}} \cdots t_{\alpha_{2}} t_{\alpha_{1}}$ of Dehn twists is called the \textit{total monodromy}.
	Conversely for a given element $\varphi$ in $\Mod (\Sigma)$, a factorization of $\varphi$ into right-handed Dehn twists is called 
	a \textit{positive factorization} of $\varphi$.
	For the monodromy $( t_{\alpha_{1}}, \,  t_{\alpha_{2}} ,\cdots, t_{\alpha_{m-1}}, \, t_{\alpha_{m}})$, we will define two transformations.
	A \textit{Hurwitz move} is defined as the partial replacement of 
	$( t_{\alpha_{1}},  t_{\alpha_{2}} ,\cdots, t_{\alpha_{i}}, t_{\alpha_{i+1}}, \cdots, t_{\alpha_{m}})$ by 
	\[
	( t_{\alpha_{1}},  t_{\alpha_{2}} ,\cdots, t_{\alpha_{i+1}}, t_{\alpha_{i+1}} t_{\alpha_{i}} t_{\alpha_{i+1}}^{-1}, \cdots, t_{\alpha_{m}}).
	\]
	A \textit{total conjugation} by an element $\psi$ of $\Mod (\Sigma)$ is defined as taking the conjugate of each $t_{\alpha_{i}}$
	by $\psi$ simultaneously so that 
	\[
	( \psi t_{\alpha_{1}} \psi^{-1},  \psi  t_{\alpha_{2}} \psi^{-1} ,\cdots, \psi  t_{\alpha_{m}} \psi^{-1}).
	\]
	If the monodromy $( t_{\alpha_{1}},  t_{\alpha_{2}} ,\cdots, t_{\alpha_{m}})$ transforms another 
	$m$-tuple $( t_{\alpha'_{1}},  t_{\alpha'_{2}} ,\cdots, t_{\alpha'_{m}})$ by Hurwitz moves and total conjugations, 
	we write 
	\[
	( t_{\alpha_{1}},  t_{\alpha_{2}} ,\cdots, t_{\alpha_{m}}) \equiv ( t_{\alpha'_{1}},  t_{\alpha'_{2}} ,\cdots, t_{\alpha'_{m}}).
	\]
	Note that the two transformations preserve the diffeomorphism types of $X$.
	The reader finds geometric meanings of the two transformations in \cite[Section $2$ and $4$]{Mat}. 
	
	\begin{definition}\label{def: PALF}
	 A Lefschetz fibration is said to be \textit{allowable} if all of the vanishing cycles are homologically non-trivial in the fiber $\Sigma$. 
	 After this, we abbreviate a positive allowable Lefschetz fibration to a \textit{PALF}.
	 \end{definition}
	 		
	Next we will review open book decompositions. 
	Let $M$ be a closed, oriented, connected $3$-manifold.
	\begin{definition}\label{def: open book}                           
		Let $L$ be an oriented link in $M$ and $\pi: (M-L) \rightarrow S^1$ a smooth map. 
		A pair $(L, \pi)$ is called an \textit{open book decomposition} of $M$
		if $\pi: (M-L) \rightarrow S^1$ is a fibration of the complement of $L$ such that 
		$\pi^{-1}(\theta )$ is the interior of a compact surface $\Sigma_{\theta}\subset M$ 
		which is diffeomorphic to a compact surface $\Sigma$ and whose boundary is $L$ for any $\theta \in S^1$.   
	\end{definition}
	
	The oriented link $L$ is called the \textit{binding} and 
	the compact surface $\Sigma$ is called the \textit{page} of the open book decomposition $(L, \pi)$.

	Once we give an open book decomposition $(L,\pi)$ with page $\Sigma$ of $M$, 
	we consider it as a pair of the surface $\Sigma$ and 
	a diffeomorphism $\varphi$ of $\Sigma$ relative to the boundary, called a \textit{monodromy} of the open book decompositions, as follows: 
	$M - L$ is diffeomorphic to the interior of a mapping torus $[0,1]\times \Sigma/ ((1,x) \sim  (0,\varphi(x))$.
	The pair $(\Sigma, \varphi)$ is often said to be an \textit{abstract open book decomposition} of $M$. 
	Conversely for any abstract open book decomposition $(\Sigma, \varphi )$, 
	there exists an open book decomposition $(L, \pi)$ (cf. \cite[Section $2$]{E}).
	After this, we do not distinguish abstract open book decompositions from open book decompositions throughout the remainder of the article.
	
	A contact $3$-manifold $(M, \xi)$ is related to an open book decomposition of $M$ as follows.
	
	\begin{definition}\label{def: support}
		A contact structure $\xi $ on $M$ is \textit{supported} by an open book decomposition $(L, \pi )$ of $M$ 
		if there exists a contact form $\alpha$ for $\xi$ such that 
		\begin{enumerate}
			\item $d\alpha$ is a positive volume form on each page $\Sigma_{\theta}$ of the open book decomposition, and
			\item $\alpha > 0$ on the binding $L$, that is, for any vector field $v$ inducing the positive orientation of $L$, $\alpha (v) > 0$.   
		\end{enumerate}
	\end{definition}

		Loi and Piergallini \cite{LP}, and Akbulut and Ozbagci \cite{AO} showed that 
		every open book decomposition with monodromy 
		which has a positive factorization supports a Stein fillable contact structure, and 
		every Stein fillable contact structure is supported 
		by an open book decomposition with monodromy which has a positive factorization.
		It can be easily check that, for any PALF $f:X \rightarrow D^2$ with fiber $\Sigma$, 
		we obtain the open book decomposition with page $\Sigma$ and monodromy the total monodromy of $f$.
		Conversely it can be shown that, for every open book decomposition $(\Sigma, \varphi)$ whose monodromy $\varphi$ has a positive factorization, 
		we obtain the PALF $f: X \rightarrow D^2$ with fiber $\Sigma$ and 
		monodromy the $m$-tuple determined by the positive factorization of $\varphi$.

	 \section{Main Results.}\label{sec3}
	    In this section, an integral homology $3$-sphere is called a homology $3$-sphere for short.
	    
	 \subsection{Handle decompositions of Stein fillings of homology $3$-spheres}
		Before stating the results in this subsection, 
		we review some conditions for a contact structure on a $3$-manifold to be supported by a planar open book decomposition.
	 
	 \begin{theorem}[Etnyre {\cite[Theorem $4.1$]{E2}}]\label{Etnyre}
	 Suppose that $X$ is a symplectic filling of a contact $3$-manifold $(M,\xi)$ supported 
	 by a planar open book decomposition.
	 Then $b_{2}^{+}(X) = b_{2}^{0}(X) = 0$, and the boundary of $X$ is connected.
	 Moreover, if $M$ is a homology $3$-sphere, then the intersection form $Q_{X}$ of $X$ is 
	 diagonalizable over the integers.
	 \end{theorem}
	 
	 \begin{theorem}[Ozv\'{a}th, Stipsicz and Szab\'{o} {\cite[Corollary $1.5$]{OSS}}]\label{OSS}
	 Suppose that a contact $3$-manifold $(M, \xi)$ with the Euler class $e(\xi) = 0$ admits 
	 a Stein filling $X$ such that the first Chern class $c_{1}(X) \neq 0$.
	 Then $\xi$ is not supported by a planar open book decomposition. 
	 \end{theorem}
	 
	 Using the above theorems, we will show the following lemma. 
	 
	 \begin{lemma}\label{H2}
	Let $M$ be a homology $3$-sphere.
	Suppose that a contact $3$-manifold $(M, \xi)$ admits a Stein filling $X$ 
	and $\xi$ is supported by a planar open book decomposition.
	Then $H_{2}(X; \Z) =0$.
	 \end{lemma}
	 
	 \begin{proof}
	 It is well-known that the Stein filling $X$ admits a handle decomposition without 
	 $3$- or $4$-handles (see \cite{El2} and \cite{G}).
	 Thus $H_{2}(X;\Z)$ has no torsion element.
	 Suppose that the rank of $H_{2}(X; \Z)$ is positive.
	 According to Theorem \ref{Etnyre}, $Q_{X}$ is negative definite and diagonalizable, 
	 so there exists an element $a$ of $H_{2}(X;\Z)$ such that its square $a^2$ is $-1$.
	 We have 
	 \[
	 c_{1}(X)(a) \equiv  w_{2}(X)(a) \equiv a^2 \equiv 1 \  (\textrm{mod}\ 2).
	 \]
	 Therefore  $c_{1}(X) \neq 0$.
	 Now since $M$ is a homology $3$-sphere and $H^{2}(M ; \Z) = 0$, $e(\xi) = 0$.
	Hence this contradicts our assumption by Theorem \ref{OSS}, and $H_{2}(X; \Z)=0$.
	 \end{proof}
	 
	  \begin{lemma}\label{H1}
	Let $M$ be a homology $3$-sphere.
	Suppose that a contact $3$-manifold $(M, \xi)$ admits a Stein filling $X$.
	Then $H_{1}(X; \Z) =0$.
	 \end{lemma}
	 
	 \begin{proof}
	 Note that $H^{3}(X; \Z) =0$ since $X$ admits a handle decomposition without $3$-handles as mentioned above.
	 The exact sequence of cohomology of the pair $(X, \del X)$ provides the following exact sequence 
	 \[
	 0= H^{3}(X, \Z) \leftarrow H^3(X, \del X; \Z) \leftarrow H^{2}(\del X; \Z) =0\, \rm{(exact)}.
	 \]
	 Thus $H_{1}(X; \Z) \cong H^{3}(X, \del X; \Z) =0$.
	\end{proof}
	 
	 From these lemmas, once we give a handle decomposition of a Stein filling of a contact $3$-manifold, 
	we can characterize the number of $1$-handles and that of $2$-handles of the handle decomposition.
	 
	 \begin{proposition}\label{prop: handle}
	 Let $M$ be a homology $3$-sphere.
	Suppose that a contact $3$-manifold $(M, \xi)$ admits a Stein filling $X$ 
	and $\xi$ is supported by a planar open book decomposition.
	Then, for any handle decomposition of $X$ with one $0$-handle and without $3$- or $4$-handles, 
	the number of $1$-handles and that of $2$-handles agree.
	 \end{proposition}
	 
	 \begin{proof}
	 Using Lemma \ref{H2} and Lemma \ref{H1}, 
	 the Euler characteristic $\chi (X)$ of $X$ is $1$.
	 On handle homology, each $i$-th chain complex $C_{i}(X)$ is generated by the $i$-handles, 
	 thus $\displaystyle C_{1}(X) \cong \Z^m$ and $\displaystyle C_{2}(X) \cong \Z^n$, 
	where $m$ is the number of $1$-handles and $n$ is that of $2$-handles of the given handle decomposition of $X$.
	Therefore we have 
	\[
	1 = \chi (X) = 1-m +n , 
	\]
	so $m=n$.
	We complete the proof of this proposition.
	 \end{proof}
	 
	 We can understand this proposition  in terms of PALFs.
	 
	\begin{proposition}\label{prop: PALF}
	Let $M$ be a homology $3$-sphere.
	Suppose that a contact $3$-manifold $(M, \xi)$ admits a Stein filling $X$ 
	and $\xi$ is supported by a planar open book decomposition.
	If $X$ admits a PALF $f:X\rightarrow D^2$ with fiber a $(n+1)$-holed sphere $\Sigma_{0,n+1}$, 
	then the number of singular fibers of $f$ is $n$.
	\end{proposition}
	
	\begin{proof}
	We can take a handle decomposition of $X$ from the data of the vanishing cycles of the PALF $f$ as in Section \ref{section: LF and OB}.
	On the handle decomposition, the number of $1$-handles is the first Betti number of $\Sigma_{0,n+1}$ i.e., $n$, 
	and that of $2$-handles is the number of the vanishing cycles.
	Thus this proposition follows from Proposition \ref{prop: handle}. 
	\end{proof}
	
	\begin{remark}
	Proposition \ref{prop: handle} and \ref{prop: PALF} are used to examine whether a Stein fillable contact structure $\xi$ 
	on a homology $3$-sphere is supported by a planar open book decomposition.
	Although these propositions are contained by Theorem \ref{Etnyre} essentially, 
	it is easier to apply our propositions to some contact $3$-manifolds.
	For example let $(M, \xi)$ be a contact homology $3$-sphere obtained from a Legendrian surgery 
	on a Legendrian link in $(S^3, \xi_{std})$.
	Let $X$ be the Stein filling of $(M, \xi)$ defined by the Legendrian surgery diagram.
	Obviously, $X$ admits a handle decomposition consisting of one $0$-handle and $2$-handles.
	Therefore we apply Proposition \ref{prop: handle} to $X$ and conclude that $\xi$ is not supported by any planar open book decomposition.
	\end{remark}
	
	\begin{proposition}
	Let $M$ be a $3$-manifold equipped with an open book decomposition $(\Sigma_{0,n+1}, \varphi)$ and 
	$\xi$ a Stein fillable contact structure on $M$ supported by $(\Sigma_{0, n+1}, \varphi)$.
	Suppose that a Stein filling $X$ of $(M, \xi)$ admits 
	a PALF $f:X\rightarrow D^2$ determined by a positive factorization of $\varphi$.
	Then $M$ is a homology $3$-sphere if and only if $H_{1}(X; \Z) =0$ and the number of singular fibers of $f$ is $n$.
	\end{proposition}
	
	\begin{proof}
	Obviously, the necessary condition follows from Lemma \ref{H1} and Proposition \ref{prop: PALF}.
	Hence we concentrate on the sufficient condition.
	Taking a handle decomposition of $X$ from the data of the vanishing cycles of the PALF $f$ as 
	\[
	X  = (D^2 \times \Sigma_{0,n+1}) \cup (\cup_{i=1}^{m}h^{(2)}_{i})
		    =   (h^{(0)} \cup (\cup_{j=1}^{n} h^{(1)}_{j})) \cup (\cup_{i=1}^{n}h^{(2)}_{i})  
	\]
	and performing surgery on the core of the $1$-handles of $X$, we get the surgered manifold $X'$.
	Let $h_{1}^{\prime (2)}, h_{2}^{\prime (2)}, \dots, h_{n}^{\prime (2)}$ be the $2$-handles provided by the surgery.
	Let $\{h_{1}^{\prime (2)}, h_{2}^{\prime (2)}, \dots, h_{n}^{\prime (2)}, h_{1}^{(2)}, h_{2}^{(2)}, \dots , h_{n}^{(2)} \}$ be a basis for $C_{2}(X')$ and 
	for $i=1,2$, $\{ h_{1}^{(i)}, h_{2}^{(i)}, \dots, h_{n}^{(i)} \}$ one for $C_{i}(X)$.
	Under these bases, the intersection form $Q_{X'}$ can be written as 
	\[
	\begin{pmatrix}
	O & A \\
	A^T & B 
	\end{pmatrix}
	\]
	where $A$ is the matrix representation of the boundary map $\del_{2}: C_{2}(X) \rightarrow C_{1}(X)$ and 
	$B$ is the linking matrix of the attaching circles of $h_{1}^{(2)}, h_{2}^{(2)}, \dots, h_{n}^{(2)}$.
	Since $H_{1}(X; \Z)=0$, $A$ is unimodular and so is $Q_{X}$.
	Therefore $\del X = M $ is a homology $3$-sphere.
	\end{proof}
	
	\subsection{Proofs of the main results}
	 
	 We define types of simple closed curves in $\Sigma_{0,n}$.
	 A simple closed curve in $\Sigma_{0,n}$ is called a \textit{boundary curve} 
	if it is parallel to a boundary component of $\Sigma_{0,n}$ and \textit{non-boundary curve} otherwise.
	 
	 \begin{proof}[Proof of Theorem \ref{thm: uniqueness}]
	 Let $X$ be a Stein filling of $(M, \xi)$.
	 According to Theorem \ref{thm: Wendl} and Proposition \ref{prop: PALF}, $X$ admits a PALF $f: X\rightarrow D^2$ with fiber $\Sigma_{0,4}$ 
	determined by a positive factorization of $\varphi = t_{\gamma} t_{\beta} t_{\alpha}$, 
	where $\alpha$, $\beta$, $\gamma$ are simple closed curves in $\Sigma_{0,4}$.
	If these curves are of non-boundary, then $H_{1}(X; \Z)$ has a torsion element, 
	and $M$ is not a homology $3$-sphere by Proposition \ref{H1}.
	Thus at least one of these curves must be a boundary curve.
	Since a Dehn twist along a boundary curve is an element of the center of $\rm{Mod}(\Sigma_{0,4})$, 
	it suffices to consider the following three cases: 
	\begin{enumerate}
	\item $\alpha$, $\beta$, and $\gamma$ are of boundary, 
	\item $\alpha$ and $\beta$ are of boundary, and $\gamma$ is of non-boundary, and 
	\item $\alpha$ is of boundary, and $\beta$ and $\gamma$ are of non-boundary.
	\end{enumerate}
	
	In the first two cases, the open book decomposition $(\Sigma_{0,4}, \varphi)$ supports the standard contact structure $\xi_{std}$ on $S^3$.
	$(S^3, \xi_{std})$ admits a unique Stein filling, $D^4$, and this theorem folds.
	Thus we concentrate on the case $(3)$ in the rest of the proof.
		
	Considering the conjugate of $\varphi= t_{\gamma} t_{\beta} t_{\alpha}$ by an appropriate mapping class if necessary, 
	we can assume that $\alpha$ and $\beta$ lie in $\Sigma_{0,4}$ as Figure \ref{fig: SCC} 
	and $\gamma$ is a simple closed curve obtained from the dotted curve in Figure \ref{fig: SCC} twisted by a diffeomorphism of $\Sigma_{0,4}$.

	We claim that $X$ is diffeomorphic to either $D^4$ or a Mazur type manifold. 
	A Kirby diagram of $X$ is drawn as Figure \ref{fig: diagram} from the data of vanishing cycles of $f$.
	We can find two cancelling $1$-handle/$2$-handle pairs in the diagram 
	and get a diagram consisting of one $0$-handle, one $1$-handle, and one $2$-handle.
	Since $\pi_{1}(X)$ has one relator and $H_{1}(X; \Z)=0$, $X$ is simply connected.
	Furthermore, since $H_{*}(X; \Z)=0$ if $*>0$, $X$ is contractible.
	Therefore $X$ is diffeomorphic to $D^4$ if $M$ is diffeomorphic to $S^3$ or a Mazur type manifold otherwise.

	Consider another factorization $t_{\gamma'} t_{\beta'} t_{\alpha'}$ of $\varphi$, and 
	let $f':X'\rightarrow D^2$ be a PALF obtained from this factorization.
	We will show that $(t_{\alpha}, t_{\beta}, t_{\gamma}) \equiv (t_{\alpha'}, t_{\beta'}, t_{\gamma'})$.
	Since $\Mod (\Sigma_{0,4})$ is isomorphic to $( \Pi_{i=1}^4 \Z[t_{\delta_{i}}]) \times (\Z [t_{a}] \times \Z[t_{b}])$ 
	where these curves lie in $\Sigma_{0,4}$ as Figure \ref{fig: 4-holed}, 
	the abelianization of $\varphi$ is 
	\[
	[t_{\delta_{1}}]+[t_{a}]+[t_{b}] \in \Mod (\Sigma_{0,4})^{ab} \cong ( \oplus_{i=1}^4 \Z[t_{\delta_{i}}]) \oplus (\Z [t_{a}] \oplus \Z[t_{b}]).
	\]
	Thus one of $t_{\alpha'}$, $t_{\beta'}$ and $t_{\gamma'}$ is $t_{\alpha}$, and 
	the rest of the mapping classes are the conjugate of $t_{\alpha}$ and $t_{\beta}$.
	Using Hurwitz moves, we can assume that $t_{\alpha}=t_{\alpha'}$.
	Moreover we can also assume that $t_{\beta}$ and $t_{\beta'}$ are conjugate and so are $t_{\gamma}$ and $t_{\gamma'}$.
	Since each Dehn twist $t_{\delta_{i}}$ is an element of the center of $\Mod (\Sigma_{0,4})$, 
	there exist mapping classes $\psi _{1}$ and $\psi_{2}$ in $\langle  t_{\gamma}, t_{\beta}\rangle < \Mod(\Sigma_{0,4})$ such that 
	$t_{\beta'} = \psi_{1} t_{\beta} \psi_{1}^{-1}$ and $t_{\gamma'} = \psi_{2} t_{\gamma} \psi_{2}^{-1}$.
		\begin{figure}[tbh]
		\begin{center}
			\begin{tabular}{c}
				\begin{minipage}{0.5\hsize}
					\begin{center}
						\includegraphics[width=150pt]{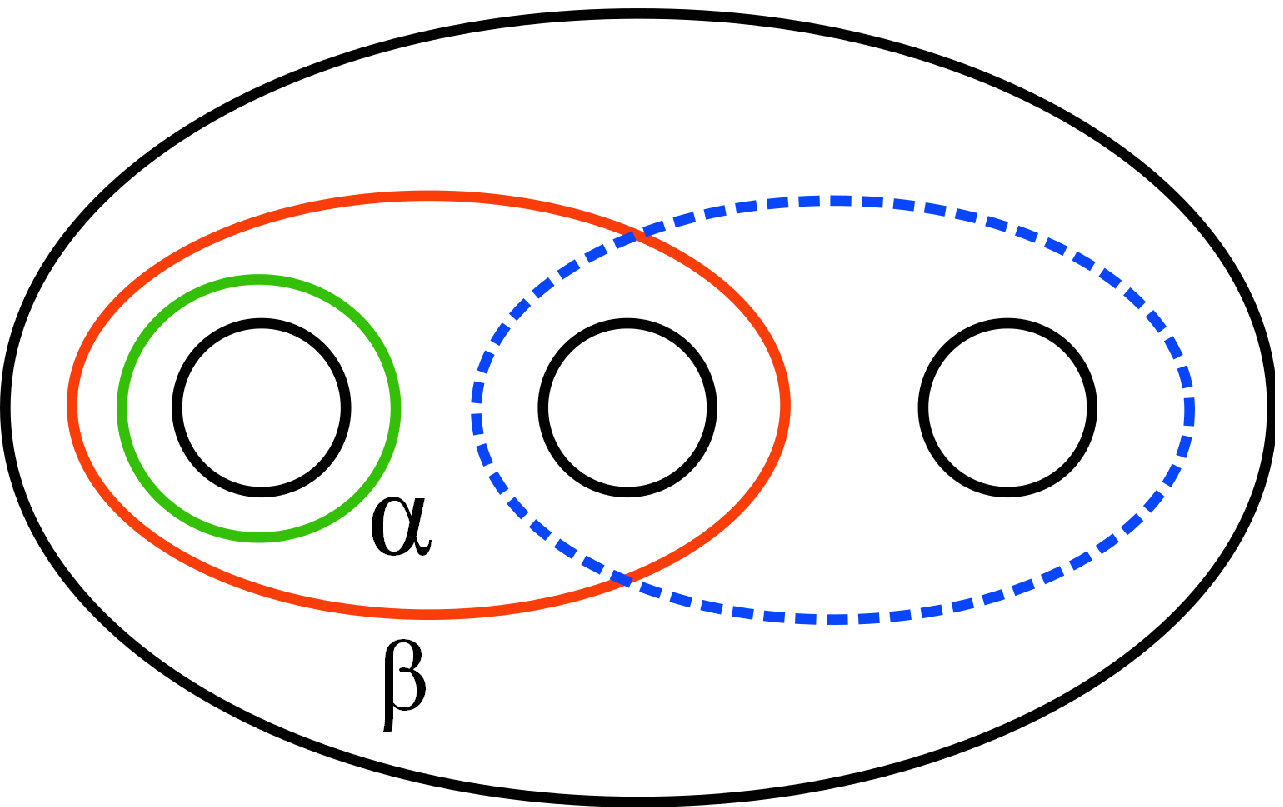}
					\end{center}
						\caption{}
						\label{fig: SCC}
				\end{minipage}
				\begin{minipage}{0.5\hsize}
					\begin{center}
						\includegraphics[width=150pt]{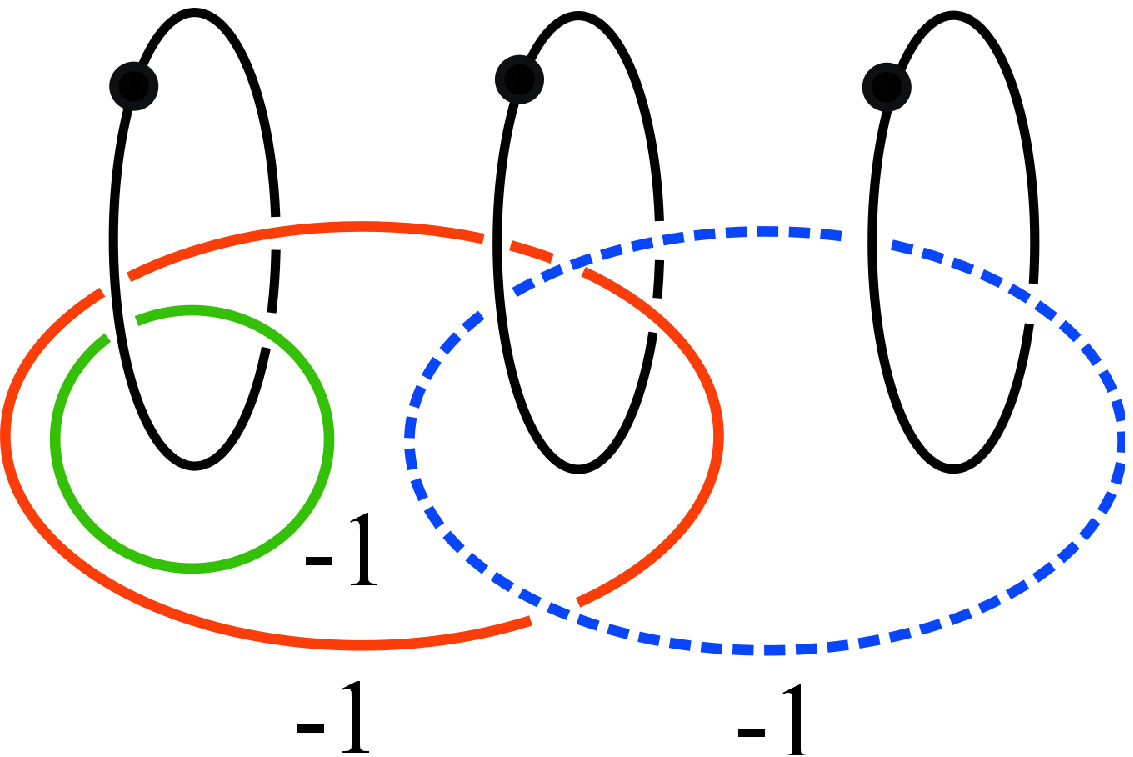}
					\end{center}
						\caption{}
						\label{fig: diagram}
				\end{minipage}
			\end{tabular}
		\end{center}
	\end{figure}
			\begin{figure}[thb]
			\begin{center}
				\includegraphics[width=150pt]{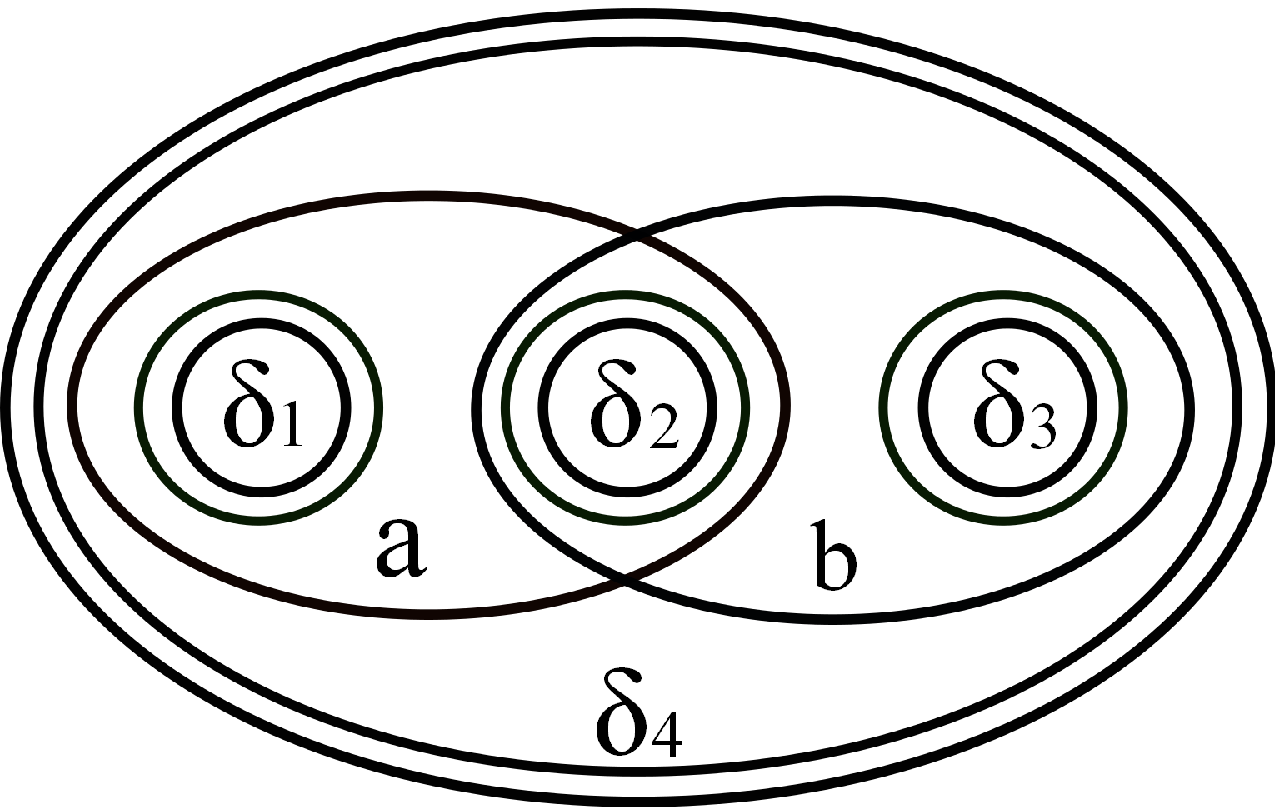}	
				\caption{}
				\label{fig: 4-holed}
			\end{center}
		\end{figure} 
	Hence we have 
	\begin{eqnarray}
	(t_{\alpha'}, t_{\beta'}, t_{\gamma'})\nonumber 	  &\equiv & (t_{\alpha}, \psi_{1} t_{\beta} \psi_{1}^{-1}, \psi_{2} t_{\gamma} \psi_{2}^{-1}) \\
	\nonumber										 &\equiv & (t_{\alpha}, t_{\beta}, \psi_{1}^{-1}\psi_{2} t_{\gamma} \psi_{2}^{-1} \psi_{1}) \\\nonumber
	\nonumber										 &\equiv & (t_{\alpha}, t_{\beta}, t_{\psi_{1}^{-1}\psi_{2}(\gamma)}).
	\end{eqnarray}
	Thus it suffices to prove that  $(t_{\beta}, t_{\gamma}) \equiv (t_{\beta}, t_{\psi_{1}^{-1}\psi_{2}(\gamma)})$.
	To prove this claim, 
	consider the representation $\varrho: \langle t_{\gamma}, t_{\beta} \rangle \rightarrow {\rm{PSL}}(2;\Z)$ defined in Section \ref{MCG}.
	It is easy to check that $i(\gamma, \beta) \geq 2$ and the pair $\{ \gamma, \beta \}$ fills $\Sigma_{0,4}$. 
	Therefore the subgroup $\langle t_{\gamma}, t_{\beta} \rangle$ is isomorphic to $F_2$ and $\varrho$ is injective.
	Let $z := i(\gamma, \beta)$ and 
	$
	\begin{pmatrix}
	a &b \\
	c &d 
	\end{pmatrix}
	:=
	\varrho (\psi_{1}^{-1}\psi_{2})$.
	Then we have
	\[
	\varrho( t_{\gamma}t_{\beta}) =
	 \begin{pmatrix}
	1-z^2 & -z \\
	z & 1 
	\end{pmatrix},
	\]
	\[
	\varrho ( t_{\psi_{1}^{-1}\psi_{2}(\gamma )}t_{\beta} )= \varrho ((\psi_{1}^{-1}\psi_{2}) t_{\gamma}(\psi_{1}^{-1}\psi_{2})^{-1}t_{\beta} )=
	\begin{pmatrix}
	1-acz-a^2z  & -a^2z \\
	(1+c^2-ac)z & 2-a^2z
	\end{pmatrix},
	\]
	and their traces are 
	\[
	{\rm{tr}}(\varrho ( t_{\gamma}t_{\beta}))=2-z^2,\  {\rm{tr}}(\varrho (\psi_{1}^{-1}\psi_{2}) t_{\gamma}(\psi_{1}^{-1}\psi_{2})^{-1}t_{\beta}))=2-a^2z^2.
	\]
	Since $ t_{\gamma}t_{\beta}$ and $t_{\psi_{1}^{-1}\psi_{2}(\gamma)}t_{\beta}$ are conjugate, 
	these traces must agree, and $a=\pm 1$.
	$\varrho (\psi_{1}^{-1}\psi_{2})$ is an element of the subgroup $\langle \varrho (t_{\gamma}), \varrho (t_{\beta}) \rangle$ of ${\rm{PSL}}(2;\Z)$, 
	hence there exist integers $k$ and $\ell$ such that $b=kz$ and $c=\ell z$.
	Therefore, 
	\begin{eqnarray}
	\varrho (\psi_{1}^{-1}\psi_{2}) \nonumber & = & 
		\begin{pmatrix}
		\pm 1 & mz \\
		nz & \pm (1+mnz^2) 
		\end{pmatrix}
	\nonumber \\
	\nonumber 
	& = & 
	\begin{pmatrix}
	1 & 0 \\
	\pm nz & 1 
	\end{pmatrix}
	\begin{pmatrix}
	1 & \pm mz \\
	0 & 1 
	\end{pmatrix}
	\nonumber \\
	\nonumber
	& = & 
	\varrho (t_{\beta}^{\pm n}) \varrho(t_{\gamma}^{\mp m})
	\nonumber \\
	\nonumber 
	& = & 
	\varrho (t_{\beta}^{\pm n} t_{\gamma}^{\mp m}). 
	\end{eqnarray}
	Since $\varrho$ is injective, $\psi_{1}^{-1} \psi_{2} = t_{\beta}^{\pm n} t_{\gamma}^{\mp m}$.
	Thus we have 
	\begin{eqnarray}
	(t_{\beta}, t_{\psi_{1}^{-1}\psi_{2}(\gamma)}) \nonumber & \equiv & ( t_{\beta}, t_{ t_{\beta}^{\pm n} t_{\gamma}^{\mp m}(\gamma)}) \nonumber \\
	\nonumber & \equiv & (t_{\beta}, t_{ t_{\beta}^{\pm n}(\gamma)}) \nonumber \\
	\nonumber & \equiv & 
	(t_{\beta}^{\mp m}t_{\beta} t_{\beta}^{\pm m}, t_{\beta}^{\mp m}t_{ t_{\beta}^{\pm m}(\gamma )}t_{\beta}^{\pm}) \nonumber \\
	\nonumber & \equiv & (t_{\beta}, t_{\gamma}),
	\end{eqnarray}
	in particular 
	\[
	(t_{\alpha}, t_{\beta}, t_{\gamma}) \equiv (t_{\alpha}, t_{\beta}, t_{\psi_{1}^{-1} \psi_{2} (\gamma)}) \equiv (t_{\alpha'}, t_{\beta'}, t_{\gamma'}).
	\]
	We conclude that the monodromies obtained from positive factorizations of $\varphi$ agree 
	up to Hurwitz move and total conjugation, and $X$ and $X'$ are mutually diffeomorphic.
	Therefore $(M, \xi)$ admits also unique Stein filling up to diffeomorphism.
	 \end{proof}
	 
	\begin{proof}[Proof of Corollary \ref{cor}]
	This corollary follows from \cite[Teorem $1.1$]{OO} immediately. 
	\end{proof}

\bibliographystyle{amsplain}

\end{document}